\documentclass[a4paper,12pt]{preprint}
\usepackage[margin=1.2in]{geometry}  
\usepackage[full]{textcomp}
\usepackage[osf]{newtxtext}

\usepackage{mathalfa}
\usepackage{microtype}

\usepackage{enumitem}
\usepackage{amsmath}
\usepackage{amsfonts} 
\usepackage{amssymb}             

\usepackage{comment}
\usepackage{breakurl}

\usepackage{mathrsfs}
\usepackage{booktabs}
\usepackage{mathtools}

\usepackage{tikz}
\usepackage{amsthm}                
\usepackage{mhequ}
\usepackage{hyperref}

\usepackage{accents}

\hypersetup{pdfstartview=}

\addtocontents{toc}{\protect\hypersetup{hidelinks}}

\DeclareSymbolFont{sfoperators}{OT1}{ptm}{m}{n}
\DeclareSymbolFontAlphabet{\mathsf}{sfoperators}

\makeatletter
\def\operator@font{\mathgroup\symsfoperators}
\makeatother

\numberwithin{equation}{section}

\newtheorem{thm}{Theorem}[section]

\newtheorem{lem}[thm]{Lemma}
\newtheorem{prop}[thm]{Proposition}

\newtheorem{cor}[thm]{Corollary}

\theoremstyle{remark}

\newtheorem{rmk}[thm]{Remark}
\newtheorem{rem}[thm]{Remark}

\makeatletter
\def\th@newremark{\th@remark\thm@headfont{\bfseries}}

\def\bdiamond{\mathop{\mathpalette\bdi@mond\relax}}
\newcommand\bdi@mond[2]{%
	\vcenter{\hbox{\m@th
			\scalebox{\ifx#1\displaystyle 2.6\else1.8\fi}{$#1\diamond$}%
	}}%
}

\def\bDiamond{\mathop{\mathpalette\bDi@mond\relax}}
\newcommand\bDi@mond[2]{%
	\vcenter{\hbox{\m@th
			\scalebox{\ifx#1\displaystyle 2.6\else1.2\fi}{$#1\Diamond$}%
	}}%
}
\makeatletter

\definecolor{darkgreen}{rgb}{0.1,0.7,0.1}
\definecolor{darkred}{rgb}{0.7,0.1,0.1}
\definecolor{darkblue}{rgb}{0,0,0.7}
\newcommand{\NN}{\mathbb{N}}
\newcommand{\PP}{\mathbb{P}}
\newcommand{\RR}{\mathbb{R}}

\newcommand{\fF}{\mathcal{F}}

\newcommand{\sS}{\mathcal{S}}

\newcommand{\eps}{\varepsilon}

\makeatletter 
\newcommand{\cov}{{\operator@font cov}}
\newcommand{\var}{{\operator@font var}}
\newcommand{\corr}{{\operator@font corr}}
\newcommand{\diam}{{\operator@font diam}}
\newcommand{\Av}{{\operator@font Av}}
\newcommand{\trig}{{\operator@font trig}}
\newcommand{\Enh}{{\operator@font Enh}}
\makeatother

\colorlet{symbols}{blue!90!black}
\colorlet{testcolor}{green!60!black}

\def\${|\!|\!|}

\makeatletter
\def\DeclareSymbol#1#2#3{\expandafter\gdef\csname MH@symb@#1\endcsname{\tikz[baseline=#2,scale=0.15,draw=symbols]{#3}}\expandafter\gdef\csname MH@symb@#1s\endcsname{\scalebox{0.7}{\tikz[baseline=#2,scale=0.15,draw=symbols]{#3}}}}
\def\<#1>{\csname MH@symb@#1\endcsname}
\makeatother

\setlist[itemize]{topsep=3pt,itemsep=1.5pt,parsep=0pt}

\def\scal#1{\langle#1\rangle}
\def\cent#1{\mathopen{{\langle\kern-0.3em\rangle}}#1\mathclose{{\langle\kern-0.3em\rangle}}}

\def\d{\partial}

\begin{document}

\title{Decay of the stochastic linear Schr\"odinger equation in $d \geq 3$ with small multiplicative noise}
\author{Chenjie Fan$^1$ and Weijun Xu$^2$}
\institute{University of Chicago, US, \email{cjfanpku@gmail.com}
\and University of Oxford, UK, \email{weijunx@gmail.com}}

\maketitle

\begin{abstract}
We give decay estimates of the solution to the linear Schr\"odinger equation in dimension $d \geq 3$ with a small noise which is white in time and colored in space. As a consequence, we also obtain certain asymptotic behaviour of the solution. The proof relies on the bootstrapping argument used by Journ\'e-Soffer-Sogge for decay of deterministic Schr\"odinger operators. 
\end{abstract}

\setcounter{tocdepth}{1}
\microtypesetup{protrusion=false}
\microtypesetup{protrusion=true}

\section{Introduction}

\subsection{Statement of the result}
Let $d \geq 3$ and $V \in \sS(\RR^d)$ be a Schwartz function. Consider the Schr\"odinger equation
\begin{equation} \label{eq:main_equation}
i \d_t \Psi + \Delta \Psi = \delta V \Psi \dot{B} - \frac{i}{2} \delta^2 V^2 \Psi, \qquad \Psi(0,\cdot) = f \in \sS(\RR^d), 
\end{equation}
where $B$ is a standard Brownian motion on the probability space $(\Omega, \fF, \PP)$ with filtration $(\fF_t)_{t \geq 0}$, the product between $\Psi$ and $\dot{B}$ is in the It\^o sense, and $\delta>0$ is a small number to be specified later. The linear correction term $-\frac{i}{2} \delta^2 V^2 \Psi$ makes the $L^2$ norm of $\Psi$ conserved pathwise. 

For every $\rho \geq 1$ and $q \geq 1$, we write $L_{\omega}^{\rho} L_{x}^{q} := L^{\rho}(\Omega,L^{q}(\RR^d))$. The main estimate is the following. 

\begin{thm} \label{th:main_decay}
	Let $\Psi$ be the solution to \eqref{eq:main_equation}. For every $\rho \geq 1$, $q \in [2,+\infty)$ and $\frac{1}{p} + \frac{1}{q} = 1$, there exists $\delta_0 > 0$ such that for all $\delta < \delta_0$, we have
	\begin{equation}\label{eq: esmain}
	\|\Psi(t)\|_{L_{\omega}^{\rho}L_{x}^{q}} \lesssim_{\rho,q} t^{-d(\frac{1}{2}-\frac{1}{q})} \|f\|_{L_x^p}
	\end{equation}
	for all $t > 0$ and $f \in \sS(\RR^d)$. The proportionality constant is independent of $t$ and $\delta$. 
\end{thm}

The estimates in Theorem~\ref{th:main_decay} allows us to start with initial data in $L^p$ for any $p \in (1,2]$. Another consequence of the decay estimate is the asymptotic behaviour of the solution. 

\begin{prop} \label{pr:scattering}
	For every $f \in L^2(\RR^d)$, there exists $g \in L_{\omega}^{\infty}L_{x}^{2}$ such that
	\begin{equation*}
	\lim_{t \rightarrow +\infty} \|\Psi(t) - e^{it\Delta} g\|_{L_{\omega}^{\rho} L_{x}^{2}} = 0
	\end{equation*}
	for every $\rho \geq 1$. 
\end{prop}

\begin{rem}
For simplicity of presentation, we choose the noise to be of the form $\dot{W}(t,x) = V(x) \dot{B}(t)$. This factorisation or finite dimensionality is not essential, and the argument still works through as long as $\dot{W}$ is white in time and sufficiently nice in space. 
\end{rem}

\begin{rem}
We shall see later that the assumption on $V$ could be relaxed. In fact, one only needs $\delta \big( \|V\|_{L_x^1} + \|\widehat{V}\|_{L_x^1} \big)$ to be small. We however write it in terms of $\delta$ times a Schwartz function to avoid appearance of various norms in the bounds later. 
\end{rem}

\subsection{Background and motivation}

It is well known (\cite{tao2006nonlinear}, \cite{cazenave2003semilinear}) that the free Schr\"odinger operator $e^{it\Delta}$ satisfies the dispersive estimate
\begin{equation} \label{eq:dispersive_free}
\|e^{it\Delta}\|_{L_x^p \rightarrow L_x^q} \lesssim t^{-d(\frac{1}{2}-\frac{1}{q})}\;, \qquad q \in [2,+\infty]\;, \quad \frac{1}{p} + \frac{1}{q} = 1
\end{equation}
in any dimension $d$. The estimate for the pair $(p,q) = (2,2)$ is an immediate consequence of the unitarity of $e^{it\Delta}$ in $L^2$. The other extreme case $(1,+\infty)$ follows from the explicit representation of the integration kernel of $e^{it\Delta}$. All intermediate situations then follow from Riesz-Thorin interpolation. 

For the deterministic linear operator $e^{it(\Delta+V)}$, which corresponds to the linear equation
\begin{equation*}
i \d_t u + \Delta u + V(x)u = 0, 
\end{equation*}
it is known that when $V$ is small and $d \geq 3$, one also has the dispersive estimate
\begin{equation} \label{eq:dispersive}
\|e^{it(\Delta+V)}\|_{L_{x}^{p} \rightarrow L_{x}^{q}} \lesssim t^{-d(\frac{1}{2}-\frac{1}{q})}\;, \qquad q \in [2,+\infty]\;, \quad \frac{1}{p} + \frac{1}{q} = 1. 
\end{equation}
This estimate corresponds to a simple special case considered in \cite{decay_deterministic}, which could be derived via perturbation around the free bound \eqref{eq:dispersive_free} for the pair $(1,+\infty)$. The reason \eqref{eq:dispersive} can hold only for $d \geq 3$ is that the bound $\|e^{it\Delta}\|_{L_x^1 \rightarrow L_x^\infty}$ for the free Schr\"odinger operator in \eqref{eq:dispersive_free} is integrable for large $t$ if and only if $d \geq 3$. 

These dispersive estimates can be used to derive Strichartz estimates via the $T^* T$ method (see for example the survey \cite{dispersive_survey_Schlag} for more details), and they are essential for the study of long time behaviour of the solution to the \textit{nonlinear} Sch\"odinger equation. 

As for the stochastic case, there have been a series of well-posedness results on the stochastic nonlinear Schr\"odinger equation with multiplicative noise on whole space (\cite{dBD, Debussche_H1, Rockner1, Rockner2, hornung2016nonlinear, snls_mass_critical, snls_subcritical_approx, snls_critical_Zhang}). However, none of the above results gives information on long time behaviour of the solution unless one puts on the noise an an extra fast decay in time (see \cite{Rockner_finite_scatter} for a scattering statement for noise of finite quadratic variation in time). A first step towards the understanding the long time behaviour of the solution to the nonlinear equation would be to establish certain decay properties of the solutions to the corresponding linear equation. This is the main motivation of the current article, and we have chosen the simplest possible model in the set-up of Theorem~\ref{th:main_decay}. 

It will also be of interest to establish a stochastic version of the Strichartz estimate for the solution $\Psi$ to \eqref{eq:main_equation}, which would be of the form
\begin{equation*}
\|\Psi\|_{L_{\omega}^{\rho}L_{t}^{q}L_{x}^{r}}\lesssim \|\Psi(0,\cdot)\|_{L_{x}^{2}}
\end{equation*}
for admissible pairs $(q,r)$ satisfying the Strichartz relation. But different from the deterministic case, since Brownian motion is not time reversible, it is not clear at this stage how such an estimate could be obtained from the dispersive estimate in Theorem~\ref{th:main_decay}.

\subsection{Outline of the proof}
\label{sec:outline}

Proposition~\ref{pr:scattering} is a simple consequence of Theorem~\ref{th:main_decay}. As for Theorem~\ref{th:main_decay}, we first note that since $\|\Psi(t)\|_{L_x^2} = \|f\|_{L_x^2}$ for all $t$ almost surely, if we can prove the statement for all sufficiently large $q$ (or equivalently, all $p$ sufficiently close to $1$), then the theorem will follow from interpolation (\cite[Theorems~5.1.1 and~5.1.2]{Bergh_interpolation}). Note that the interpolation statements used here are for mixed norm spaces, and are more general than Riesz-Thorin. We refer to \cite{calderon1963intermediate}, \cite{calderon1964intermediate} and \cite{calderon1966spaces} for more details. 

It then remains to prove Theorem~\ref{th:main_decay} for $q$ sufficiently large such that
\begin{equation} \label{eq:range_q_intro}
d \Big( \frac{1}{2} - \frac{1}{q} \Big) > 1. 
\end{equation}
This is always possible since $d \geq 3$. We follow the strategy in \cite{decay_deterministic} and use a bootstrap argument. The key is to establish a bootstrap relation
\begin{equation*}
t^{d(\frac{1}{2}-\frac{1}{q})} \|\Psi(t)\|_{L_{\omega}^{\rho}L_{x}^{q}} \leq C_{1}(\rho,q) \|f\|_{L_x^p} + C_{2}(\rho,q,\delta) \sup_{r \in [0,t]} \Big( r^{d(\frac{1}{2}-\frac{1}{q})} \|\Psi(r)\|_{L_{\omega}^{\rho}L_{x}^{q}} \Big)
\end{equation*}
for some $C_{1}, C_{2}$ independent of $t$, and $C_{2}(\rho,q,\delta) \rightarrow 0$ as $\delta \rightarrow 0$. This would allow us to absorb the second term on the right hand side into the left, and obtain the claim. 

Throughout, we will frequently use the following Burkholder inequality (\cite{Burkholder,BP,Brzezniak}) to control the stochastic integral. 

\begin{prop} [Burkholder]
	\label{pr:Burkholder}
	Let $\Phi$ be progressively measurable with respect to $(\fF_t)$. Then for every $q \in [1, +\infty)$ and every $\rho \geq 2$, we have
	\begin{equation*}
	\bigg\| \int_{0}^{t} e^{i(t-s)\Delta} \Phi(s) {\rm d} B_s\bigg\|_{L_{\omega}^{\rho}L_{x}^{q}} \lesssim_{\rho,q} \bigg\| \int_{0}^{t} \|e^{i(t-s)\Delta} \Phi(s)\|_{L_x^q}^{2} {\rm d}s \bigg\|_{L_{\omega}^{\rho/2}}^{\frac{1}{2}}. 
	\end{equation*}
	As a consequence, by triangle inequality, we have
	\begin{equation*}
	\bigg\| \int_{0}^{t} e^{i(t-s)\Delta} \Phi(s) {\rm d} B_s\bigg\|_{L_{\omega}^{\rho}L_{x}^{q}} \lesssim_{\rho,q} \bigg( \int_{0}^{t} \|e^{i(t-s)\Delta} \Phi(s)\|_{L_\omega^\rho L_x^q}^{2} {\rm d}s \bigg)^{\frac{1}{2}}. 
	\end{equation*}
\end{prop}

\begin{rem} \label{rmk:rho}
	In order to apply Proposition~\ref{pr:Burkholder} directly, we will restrict to $\rho \geq 2$ below. As for Theorem~\ref{th:main_decay} and Proposition~\ref{pr:scattering}, the case $\rho \in [1,2)$ follows from that of $\rho \geq 2$. 
\end{rem}

\begin{rem}
	The reason the end-point case $(p,q)=(1,+\infty)$ is excluded from Theorem~\ref{th:main_decay} is that the Burkholder inequality does not hold for the space $L_{x}^{\infty}$. 
\end{rem}

\subsection*{Organisation of the article}

The rest of the article is organised as follows. In Section~\ref{sec:proof_scattering}, we show how Proposition~\ref{sec:proof_scattering} follows from Theorem~\ref{th:main_decay}. In Section~\ref{sec:preliminary}, we give some preliminary lemmas needed for the proof of Theorem~\ref{th:main_decay}. Section~\ref{sec:proof_main} is devoted to the proof of the main theorem in $d=3$ and $q$ sufficiently large. We briefly explain in Section~\ref{sec:high_dim} how the arguments could be modified to cover higher dimensions. By interpolation, this completes the proof of Theorem~\ref{th:main_decay}.

\subsection*{Notation}
We fix $d \geq 3$. We use $L_{x}^{q}$ to denote $L^{q}(\RR^d)$, and write $L_{\omega}^{\rho} L_{x}^{q} = L^{\rho}\big( \Omega, L^{q}(\RR^d) \big)$. For any $r \in \RR$, we write $\scal{r} = 1 + |r|$. Also, since the statements are for every fixed pair $(p,q)$, we use $\alpha$ to denote
\begin{equation*}
\alpha = d \Big( \frac{1}{2} - \frac{1}{q} \Big). 
\end{equation*}
When non-commutative products are involved, we write
\begin{equation} \label{eq:product_non_comm}
\prod_{j=1}^{m} A_{j} = A_{m} \cdots A_{1}. 
\end{equation}
Typically these $A_j$'s will be operators. Finally, according to Remark~\ref{rmk:rho}, we assume without loss of generality that $\rho \geq 2$.

\subsection*{Acknowledgment}
We thank Zihua Guo and Carlos Kenig for discussion, in particular on interpolation. WX gratefully acknowledges the support from the Engineering and Physical Sciences Research Council through the fellowship EP/N021568/1. 

\section{Proof of Proposition~\ref{pr:scattering}}
\label{sec:proof_scattering}

Since $e^{it\Delta}$ is unitary and $\Psi$ has pathwise mass conservation, it suffices to prove the proposition for $f \in \sS(\RR^d)$. We need to show that $e^{-it\Delta} \Psi(t)$ has a limit in $L_{\omega}^{\rho}L_{x}^{2}$ as $t \rightarrow +\infty$, and equivalently, $\{e^{-it\Delta} \Psi(t)\}_{t}$ is Cauchy in $L_{\omega}^{\rho} L_{x}^{2}$. To see this, we write down the Duhamel formula
\begin{equation} \label{eq:scattering_duhamel}
e^{-it\Delta} \Psi(t) - e^{-is\Delta} \Psi(s) = -i \delta \int_{s}^{t} e^{-ir\Delta} \big( V \Psi(r) \big) {\rm d}B_r - \frac{\delta^2}{2} \int_{s}^{t} e^{-ir\Delta} \big( V^2 \Psi(r) \big) {\rm d}r. 
\end{equation}
We need to control the $L_{\omega}^{\rho}L_{x}^{2}$-norm of the two terms on the right hand side. For the first one, by Burkholder and triangle inequalities (Proposition~\ref{pr:Burkholder}), we have
\begin{equation*}
\bigg\| \int_{s}^{t} e^{-ir\Delta} \big(V\Psi(r)\big) {\rm d} B_r \bigg\|_{L_{\omega}^{\rho}L_{x}^{2}}^{2} \lesssim_{\rho} \int_{s}^{t} \big\| e^{-ir\Delta} \big( V \Psi(r) \big)  \big\|_{L_{\omega}^{\rho}L_{x}^{2}}^{2} {\rm d}r. 
\end{equation*}
The integrand on the right hand side then can be controlled as
\begin{equation*}
\|e^{-ir\Delta} \big( V \Psi(r) \big)\|_{L_{\omega}^{\rho}L_{x}^{2}}^{2} = \|V \Psi(r)\|_{L_{\omega}^{\rho}L_{x}^{2}}^{2} \lesssim_{V} \|\Psi(r)\|_{L_{\omega}^{\rho}L_{x}^{q}}^{2} \lesssim_{V,\rho} r^{-2d(\frac{1}{2}-\frac{1}{q})} \|f\|_{L_x^p}^{2}, 
\end{equation*}
where we have used the unitary property of $e^{-ir\Delta}$ and the decay estimates in Theorem~\ref{th:main_decay}, and the bound in the middle follows from $q \geq 2$ and H\"older. Note that the bound above holds for every $q \geq 2$. If we choose $q$ sufficiently large such that \eqref{eq:range_q_intro} holds, one can immediately deduce that
\begin{equation*}
\bigg\| \int_{s}^{t} e^{-ir\Delta} \big(V\Psi(r)\big) {\rm d} B_r \bigg\|_{L_{\omega}^{\rho}L_{x}^{2}}^{2} \rightarrow 0
\end{equation*}
as $s, t \rightarrow +\infty$. The second term on the right hand side of \eqref{eq:scattering_duhamel} can be controlled in a similar way with $q$ in the range of \eqref{eq:range_q_intro}. One can see the limit belongs to $L_{\omega}^{\infty}L_{x}^{2}$ since $e^{-it\Delta}$ is unitary and the $L^2$-norm of $\Psi(t)$ is conserved. 

The situation for general $f$ follows from the unitarity of $e^{it\Delta}$ and pathwise mass conservation. This completes the proof of Proposition~\ref{pr:scattering}. 

\begin{rmk}
	One can see from the proof that the Cauchy property for the stochastic term only needs $q$ to satisfy
	\begin{equation*}
	2d \Big( \frac{1}{2} - \frac{1}{q} \Big) > 1, 
	\end{equation*}
	while control of the third term on the right hand side of \eqref{eq:scattering_duhamel} requires $q$ to satisfy \eqref{eq:range_q_intro}.  
\end{rmk}

\section{Preliminary bounds}
\label{sec:preliminary}

We give some preliminary bounds that will be used in the rest of the article. Throughout, $\Psi$ denotes the solution to \eqref{eq:main_equation} with initial data $f$. We also fix arbitrary $q \in [2,+\infty)$ and $p$ such that $\frac{1}{p} + \frac{1}{q} = 1$. Recall
\begin{equation*}
\alpha = d \Big(\frac{1}{2}-\frac{1}{q}\Big)
\end{equation*}
and the dispersive estimate for the free Schr\"odinger operator
\begin{equation} \label{eq:free_dispersive}
\|e^{it\Delta}\|_{L_{x}^{p} \rightarrow L_{x}^{q}} \lesssim |t|^{-\alpha}\;, \qquad t \in \RR. 
\end{equation}
We do not impose conditions on $\alpha$ in this section. Also recall \eqref{eq:product_non_comm} for the non-commutative product. We have the following lemma. 

\begin{lem} \label{le:exchange_trig}
	For every $m \in \NN$ and every $u_{1}, \dots, u_{m} \in \RR$ with $\sum_{j} u_{j} \neq 0$, there exists $\theta, \eta, \zeta \in \RR$ such that
	\begin{equation} \label{eq:exchange_trig}
	e^{i u_{m} \Delta} \prod_{j=1}^{m-1} \Big( e^{i \scal{\xi_j, \cdot}} e^{i u_{j} \Delta}\Big) = e^{i \theta} e^{i \scal{\eta, \cdot}} e^{i \big( \sum_{j=1}^{m} u_j \big) \Delta} e^{i \scal{\zeta, \cdot}}. 
	\end{equation}
	As a consequence, we have
	\begin{equation} \label{eq:exchange_V}
	\Big\| e^{i u_m \Delta} \prod_{j=1}^{m-1} \big(V_{j} e^{i u_{j}\Delta} \big) \Big\|_{L_{x}^{p} \rightarrow L_{x}^{q}} \lesssim \Big| \sum_{j=1}^{m} u_j \Big|^{-\alpha} \prod_{j=1}^{m} \|\widehat{V_j}\|_{L_x^1}. 
	\end{equation}
\end{lem}
\begin{proof}
	The assertion~\eqref{eq:exchange_trig} is precisely \cite[Lemma~2.4]{decay_deterministic}. The second assertion~\eqref{eq:exchange_V} is a direct consequence of the first one and the dispersive estimate \eqref{eq:free_dispersive} for the free Schr\"odinger operator.  
\end{proof}

If we know a lower bound of $|\sum_{j} u_j|$, we may improve the above lemma to the following. 

\begin{lem} \label{le:op_strong_decay}
	For every $\eps>0$, we have
	\begin{equation} \label{eq:op_strong_decay}
	\bigg\| e^{i u_m \Delta} \prod_{j=1}^{m} \big( V_j e^{i u_j \Delta} \big) \bigg\|_{L_x^p \rightarrow L_x^q} \lesssim_{\eps,m} \bigg( \prod_{j=1}^{m} \scal{u_j}^{-\alpha} \bigg)  \prod_{j=1}^{m-1} \Big( \|\widehat{V_j}\|_{L^1} + \|V_j\|_{L^{\frac{pq}{q-p}}} \Big), 
	\end{equation}
	uniformly over the points $u_1, \dots, u_m \in \RR$ such that $|\sum_{j} u_j| > \eps$. 
\end{lem}
\begin{proof}
	This is the content of the first bound in \cite[Lemma~2.6]{decay_deterministic}. 
\end{proof}

We now look at bounds that involve $\Psi$, the solution to \eqref{eq:main_equation}. We now require all the times $u_j$ involved are positive. 

\begin{lem} \label{le:Gronwall}
	For $0 \leq s \leq t$, let
	\begin{equation*}
	F_{t}(s) = \sup_{\xi \in \RR^d} \big\| e^{i(t-s)\Delta} e^{i \scal{\xi, \cdot}} \Psi(s) \big\|_{L_{\omega}^{\rho}L_{x}^{q}}. 
	\end{equation*}
	Then, there exists $C=C(\rho,q)$ such that
	\begin{equation*}
	F_{t}(s) \leq C \cdot t^{-\alpha} \|f\|_{L_{x}^{p}} \exp\Big(C \big(\delta^2 \|\widehat{V}\|_{L^1}^{2} s + \delta^4 \|\widehat{V}\|_{L^1}^{4} s^2 \big) \Big). 
	\end{equation*}
	for all $t>0$ and all $s \in [0,t]$. 
\end{lem}
\begin{proof}
	We expand $\Psi(s)$ as
	\begin{equation*}
	\Psi(s) = e^{is\Delta}f - i \delta \int_{0}^{s} e^{i(s-r)\Delta}\big( V \Psi(r) \big) {\rm d} B_r - \frac{\delta^2}{2} \int_{0}^{s} e^{i(s-r)\Delta} \big( V^2 \Psi(r) \big) {\rm d} r. 
	\end{equation*}
	Hence, we have the bound
	\begin{equation*}
	F_{t}(s) \leq \text{(I)} + \text{(II)} + \text{(III)}, 
	\end{equation*}
	where
	\begin{equation*}
	\begin{split}
	\text{(I)} &= \sup_{\xi \in \RR^d} \big\| e^{i(t-s)\Delta} e^{i \scal{\xi, \cdot}} e^{is\Delta} f \big\|_{L_{\omega}^{\rho}L_{x}^{q}},\\
	\text{(II)} &= \delta \sup_{\xi \in \RR^d} \Big\| \int_{0}^{s} e^{i(t-s)\Delta} e^{i \scal{\xi, \cdot}} e^{i(s-r)\Delta} V \Psi(r) {\rm d} B_r \Big\|_{L_{\omega}^{\rho}L_{x}^{q}},\\
	\text{(III)} &= \frac{\delta^2}{2} \sup_{\xi \in \RR^d} \Big\| \int_{0}^{s} e^{i(t-s)\Delta} e^{i \scal{\xi, \cdot}} e^{i(s-r)\Delta} V^2 \Psi(r) {\rm d} r \Big\|_{L_{\omega}^{\rho}L_{x}^{q}}. 
	\end{split}
	\end{equation*}
	We control them one by one. For the first one, a direction application of Lemma~\ref{le:exchange_trig} shows that it is controlled by
	\begin{equation*}
	\text{(I)} \leq C t^{-\alpha} \|f\|_{L_{x}^{p}}. 
	\end{equation*}
	For the third one, another application of \eqref{eq:exchange_trig} shows that the integrand satisfies
	\begin{equation*}
	\begin{split}
	\big\| e^{i(t-s)\Delta} e^{i \scal{\xi,\cdot}} e^{i(s-r)\Delta} V^2 \Psi(r) \big\|_{L_{\omega}^{\rho}L_{x}^{q}} &= \Big\| \int_{\RR^d} \widehat{V^2}(\eta) e^{i(t-s)\Delta} e^{i \scal{\xi,\cdot}} e^{i(s-r)\Delta} e^{i \scal{\eta, \cdot}} \Psi(r) {\rm d} \eta \Big\|_{L_{\omega}^{\rho}L_{x}^{q}}\\
	&\leq \|\widehat{V}\|_{L^1}^{2} \sup_{\eta \in \RR^d} \big\| e^{i(t-r)\Delta} e^{i\scal{\eta,\cdot}}\Psi(r) \big\|_{L_{\omega}^{\rho}L_{x}^{q}}\\
	&= \|\widehat{V}\|_{L^1}^{2} F_{t}(r), 
	\end{split}
	\end{equation*}
	which is true for all $\xi \in \RR^d$. Hence, we have
	\begin{equation*}
	\text{(III)} \leq \frac{\delta^2}{2} \|\widehat{V}\|_{L^1}^{2} \int_{0}^{s} F_{t}(r) {\rm d}r. 
	\end{equation*}
	As for (II), by Burkholder and then triangle inequalities, we have
	\begin{equation*}
	\begin{split}
	\text{(II)} &\leq C \delta \sup_{\xi \in \RR^d} \Big( \int_{0}^{s} \big\| e^{i(t-s)\Delta} e^{i \scal{\xi, \cdot}} e^{i(s-r)\Delta} V \Psi(r) \big\|_{L_{\omega}^{\rho} L_{x}^{q}}^{2} {\rm d} r \Big)^{\frac{1}{2}}\\
	&\leq C \delta \|\widehat{V}\|_{L^1} \Big(\int_{0}^{s} \big(F_{t}(r)\big)^{2} {\rm d} r \Big)^{\frac{1}{2}}, 
	\end{split}
	\end{equation*}
	where we used the previous bound for (III) with $\widehat{V^2}$ replaced by $\widehat{V}$. Combining the above three bounds, we have
	\begin{equation*}
	F_{t}(s) \leq C \Big[ t^{-\alpha} \|f\|_{L_{\omega}^{\rho}L_{x}^{p}} + \delta \|\widehat{V}\|_{L^1} \Big( \int_{0}^{s} \big( F_{t}(r) \big)^{2} {\rm d} r \Big)^{\frac{1}{2}} + \delta^2 \|\widehat{V}\|_{L^1}^{2} \int_{0}^{s} F_{t}(r) {\rm d} r \Big]. 
	\end{equation*}
	Let $K_{t}(s) = \big( F_{t}(s) \big)^{2}$, the above bound then implies
	\begin{equation*}
	K_{t}(s) \leq C \Big( t^{-2\alpha} \|f\|_{L_{\omega}^{\rho}L_{x}^{p}}^{2} + \delta^2 \|\widehat{V}\|_{L^1}^{2} \int_{0}^{s} K_{t}(r) {\rm d} r + s \delta^4 \|\widehat{V}\|_{L^1}^{4} \int_{0}^{s} K_{t}(r) {\rm d} r \Big), 
	\end{equation*}
	where the additional factor $s$ in front of the last integral comes from H\"older. The desired bound then follows from Gr\"onwall and then taking square root of $K_t$. 
\end{proof}

Combining \eqref{eq:exchange_trig} and Lemma~\ref{le:Gronwall}, we have the following consequence. 

\begin{cor} \label{cor:small_time}
	For $m \geq 1$ and $u_{1}, \dots, u_{m} \in \RR^{+}$, we have
	\begin{equation*}
	\begin{split}
	&\phantom{111}\Big\| e^{i u_{m} \Delta} V_{m-1} \cdots V_{2} e^{i u_{2}\Delta} V_{1} \Psi(u_1) \Big\|_{L_{\omega}^{\rho}L_{x}^{q}}\\
	&\leq C \Big( \sum_{j=1}^{m} u_{j} \Big)^{-\alpha} \|f\|_{L_{x}^{p}} \cdot \Big( \prod_{j=1}^{m-1} \|\widehat{V_j}\|_{L^1} \Big) \exp\Big(C \big(\delta^2 \|\widehat{V}\|_{L^1}^{2} u_1 + \delta^4 \|\widehat{V}\|_{L^1}^{4} u_1^2 \big) \Big), 
	\end{split}
	\end{equation*}
	where $C = C(\rho,p,d)$ is independent of $m$ and the $u_j$'s. 
\end{cor}

The following lemma is analogous to Lemma~\ref{le:op_strong_decay} but with $\Psi$ involved and requiring all $u_j$'s being positive. 

\begin{lem} \label{le:variable_small}
	Let $m \geq 2$. We have
	\begin{equation*}
	\begin{split}
	&\phantom{111}\bigg\| e^{iu_m \Delta} V_{m-1} \cdots V_{2} e^{i u_2 \Delta} V_1 \Psi(u_1) \bigg\|_{L_{\omega}^{\rho}L_{x}^{q}}\\
	&\lesssim_{m} \Big( \prod_{j=1}^{m} \scal{u_j}^{-\alpha} \Big) \Big( \prod_{j=1}^{m-1} \big( \|\widehat{V_j}\|_{L_x^1} + \|V_j\|_{L_{x}^{\frac{pq}{q-p}}} \big) \Big) \|f\|_{L_x^p}, 
	\end{split} 
	\end{equation*}
	uniformly over the points $u_1, \dots, u_m \in \RR^{+}$ satisfying $\sum_{j} u_j > 2$ and $u_1<1$. 
\end{lem}
\begin{proof}
	Note that the claim does not follow directly from Corollary~\ref{cor:small_time} since we do not want to have a singularity in $u_1$ when it is small. This lemma is similar to the second bound in \cite[Lemma~2.6]{decay_deterministic} but not identical, so we give detailed arguments here. 
	
	We distinguish two cases: $u_2>2^{2m}$ and $u_2 \leq 2^{2m}$, starting with the first one. In this case, let $k_{1} = 2$, and define recursively
	\begin{equation*}
	k_{\ell+1} := m \wedge \inf \big\{ j>k_{\ell}: u_{j} > 2^{2m-j} \big\}
	\end{equation*}
	until it reaches $m$. This gives an increasing (finite) sequence of integers $2 = k_1 < \dots < k_{N+1} = m$, and partitions $\{u_j\}_{j=2}^{m}$ into $N$ blocks of the form $(u_{k_\ell}, \dots, u_{k_{\ell+1}-1})$ for $\ell = 1, \dots, N$. In each block, we have
	\begin{equation} \label{eq:time_block}
	\bigg( \sum_{j=k_{\ell}}^{k_{\ell+1}-1} u_{j} \bigg)^{-\alpha} \lesssim_{m} \prod_{j=k_{\ell}}^{k_{\ell+1}-1} \scal{u_{j}}^{-\alpha} \qquad \text{and} \qquad \bigg( \sum_{j=1}^{k_2-1} u_{j} \bigg)^{-\alpha} \lesssim_{n} \prod_{j=1}^{k_2-1} \scal{u_{j}}^{-\alpha}. 
	\end{equation}
	Hence, for $\ell \geq 2$, we have the bound
	\begin{equation} \label{eq:op_strong_decay_block}
	\begin{split}
	&\phantom{111}\Big\| e^{i u_{k_{\ell+1}-1}\Delta} \Big( \prod_{j=k_{\ell}-2}^{k_{\ell+1}-2} V_{j} e^{iu_{j}\Delta} \Big) V_{k_{\ell}-1} \Big\|_{L_{x}^{q} \rightarrow L_{x}^{q}}\\
	&\lesssim_{m} \Big( \prod_{j=k_{\ell}-2}^{k_{\ell+1}-1} \scal{u_j}^{-\alpha} \Big) \prod_{j=k_{\ell}-1}^{k_{\ell+1}-2} \Big( \|\widehat{V_j}\|_{L_x^1} + \|V_{j}\|_{L_{x}^{\frac{pq}{q-p}}} \Big). 
	\end{split}
	\end{equation}
	As for $\ell=1$, Fourier expanding $V_1, \dots, V_{k_2-2}$ and applying \eqref{eq:exchange_trig}, Corollary~\ref{cor:small_time} and the second bound in \eqref{eq:time_block}, we have
	\begin{equation} \label{eq:decay_first_block}
	\begin{split}
	&\phantom{111}\Big\| e^{i u_{k_2-1} \Delta} V_{k_2-2} \cdots V_{2} e^{i u_2 \Delta} V_1 \Psi(u_1) \Big\|_{L_{\omega}^{\rho}L_{x}^{q}}\\
	&\lesssim_{m} \Big( \prod_{j=1}^{k_2-1} \scal{u_j}^{-\alpha} \Big) \Big( \prod_{j=1}^{k_2 - 1} \|\widehat{V_j}\|_{L_x^1} \Big) \|f\|_{L_x^p}. 
	\end{split}
	\end{equation}
	Muliplying \eqref{eq:decay_first_block} and \eqref{eq:op_strong_decay_block} gives the desired bound in the case $u_2 > 2^{2m}$. 
	
	We now turn to the case $u_2 \leq 2^{2m}$. If $u_j \leq 2^{2m+j}$ for all $j$, then we have
	\begin{equation} \label{eq:time_block_whole}
	\bigg( \sum_{j=1}^{m} u_{j} \bigg)^{-\alpha} \lesssim_{m} \prod_{j=1}^{m} \scal{u_j}^{-\alpha}, 
	\end{equation}
	and the bound follows. If not, then let
	\begin{equation*}
	k = \inf \big\{ j \geq 3:  u_j > 2^{2m+j} \big\}. 
	\end{equation*}
	If $k=m$, or $k<m$ and $|u_j| \leq 1$ for all $k < j \leq m$, then we still have \eqref{eq:time_block_whole}, and the bound follows in the same way. Otherwise, let
	\begin{equation*}
	k' = \inf \big\{ j>k: u_j > 1 \big\}. 
	\end{equation*}
	Then $k' < m$. Since $|u_{k'}| > 1$, by Lemma~\ref{le:op_strong_decay}, we have
	\begin{equation*}
	\begin{split}
	&\phantom{111}\Big\| e^{i u_m \Delta} \Big( \prod_{j=k'}^{m-1} V_{j} e^{i u_{j} \Delta} \Big) V_{k'-1} \Big\|_{L_{x}^{q} \rightarrow L_{x}^{q}}\\
	&\lesssim_{m} \bigg( \prod_{j=k'}^{m} \scal{u_j}^{-\alpha} \bigg)  \prod_{j=k'-1}^{m-1} \Big( \|\widehat{V_j}\|_{L^1} + \|V_j\|_{L^{\frac{pq}{q-p}}} \Big). 
	\end{split}
	\end{equation*}
	For the term $\| e^{i u_{k'-1} \Delta} V_{u_{k'-2}} \cdots V_{2} e^{i u_2 \Delta} V_{1} \Psi(u_1) \|_{L_{\omega}^{\rho}L_{x}^{q}}$, one can control it in the same way as \eqref{eq:decay_first_block} since we have
	\begin{equation*}
	\bigg( \sum_{j=1}^{k'-1} u_j \bigg)^{-\alpha} \lesssim_{m} \prod_{j=1}^{k'-1} \scal{u_j}^{-\alpha}. 
	\end{equation*}
	This completes the proof. 
\end{proof}

\section{Proof of Theorem~\ref{th:main_decay}}
\label{sec:proof_main}

We are now ready to prove Theorem~\ref{th:main_decay}. We consider $d=3$ in this section, and will briefly explain in the next section how the situation for general $d$ can be established with minor modification of the argument. Recall the notation
\begin{equation*}
\alpha = d \Big( \frac{1}{2} - \frac{1}{q} \Big)\;, \qquad d=3. 
\end{equation*}
As mentioned in Section~\ref{sec:outline}, it suffices to consider (arbitrary) $q$ in the range \eqref{eq:range_q_intro}. Also since $q < +\infty$, we have
\begin{equation} \label{eq:range_alpha_3d}
1 < \alpha < \frac{3}{2}. 
\end{equation}
Theorem~\ref{th:main_decay} is equivalent to the bound
\begin{equation} \label{eq:main_aim}
\sup_{t \in \RR^{+}} \Big( t^{\alpha} \|\Psi(t)\|_{L_{\omega}^{\rho}L_{x}^{q}} \Big) \lesssim_{\rho,q} \|f\|_{L_x^p}. 
\end{equation}
To achieve this, we need to establish an inequality of the form
\begin{equation} \label{eq:bootstrap}
t^{\alpha} \|\Psi(t)\|_{L_{\omega}^{\rho}L_{x}^{q}} \leq C_{1}(\rho,q) \|f\|_{L_x^p} + C_{2}(\rho,q,\delta) \sup_{r \in [0,t]} \Big( r^{\alpha} \|\Psi(r)\|_{L_{\omega}^{\rho}L_{x}^{q}} \Big)
\end{equation}
for all $t$, where both $C_{1}$ and $C_{2}$ are independent of $t$, and $C_{2}(\rho,q,\delta) \rightarrow 0$ as $\delta \rightarrow 0$ for every fixed $\rho$ and $q$. This will allow us to absorb the second term on the right hand side into the left and establish \eqref{eq:main_aim}. 

By Corollary~\ref{cor:small_time}, the bound \eqref{eq:bootstrap} is true with $C_{2}=0$ if $t \leq 2$. Hence, we only need to prove \eqref{eq:bootstrap} for $t>2$. Following the strategy in \cite{decay_deterministic}, we expand $\Psi$ with Duhamel formula twice to obtain
\begin{equation} \label{eq:duhamel_2}
\begin{split}
\Psi(t) = &e^{it\Delta}f - i \delta \int_{0}^{t} e^{i(t-s)\Delta} V e^{is\Delta} f{\rm d} B_s - \frac{\delta^2}{2} \int_{0}^{t} e^{i(t-s)\Delta} V^{2} e^{is\Delta} f {\rm d}s\\
&+ \text{(I)} + \text{(II)} + \text{(III)} + \text{(IV)}, 
\end{split}
\end{equation}
where
\begin{equation} \label{eq:duhamel_list}
\begin{split}
\text{(I)} &= \frac{\delta^4}{4} \int_{0}^{t} \int_{0}^{s} e^{i(t-s)\Delta} V^{2} e^{i(s-r)\Delta} V^{2} \Psi(r) {\rm d}r {\rm d}s,\\
\text{(II)} &= \frac{i \delta^3}{2} \int_{0}^{t} \int_{0}^{s} e^{i(t-s)\Delta} V^{2} e^{i(s-r)\Delta} V \Psi(r) {\rm d}B_r {\rm d}s,\\
\text{(III)} &= \frac{i \delta^3}{2} \int_{0}^{t} \int_{0}^{s} e^{i(t-s)\Delta} V e^{i(s-r)\Delta} V^2 \Psi(r) {\rm d}r {\rm d}B_s,\\
\text{(IV)} &= - \delta^2 \int_{0}^{t} \int_{0}^{s} e^{i(t-s)\Delta} V e^{i(s-r)\Delta} V \Psi(r) {\rm d}B_r {\rm d}B_s. 
\end{split}
\end{equation}
We need to show that each term in the right hand side of \eqref{eq:duhamel_2} is bounded by the right hand side of \eqref{eq:bootstrap}. Since $\rho$ and $q$ are fixed, for simplicity of notation, in what follows we write ``$\lesssim$" instead of ``$\lesssim_{\rho,q}$". Also, we restrict to $t>2$ from now on. 

\subsection{``Constant" terms}

We first treat the terms without $\Psi(r)$, that is, those on the first line of the right hand side of \eqref{eq:duhamel_2}. The bound for $e^{it\Delta} f$ is the standard dispersive estimate for $e^{it\Delta}$. 

As for the second term, since $\rho \geq 2$, by Burkholder and triangle inequalities (Proposition~\ref{pr:Burkholder}), we have
\begin{equation*}
\bigg\| \int_{0}^{t} e^{i(t-s)\Delta} V e^{is\Delta} f {\rm d} B_s \bigg\|_{L_{\omega}^{\rho}L_{x}^{q}} \lesssim \bigg( \int_{0}^{t} \|e^{i(t-s)\Delta} V e^{is\Delta} f\|_{L_{\omega}^{\rho}L_{x}^{q}}^{2} {\rm d}s \bigg)^{\frac{1}{2}}. 
\end{equation*}
Since $t>2$, by Lemme~\ref{le:op_strong_decay}, the integrand satisfies the bound
\begin{equation*}
\|e^{i(t-s)\Delta} V e^{is\Delta} f\|_{L_x^q} \lesssim \scal{t-s}^{-\alpha} \scal{s}^{-\alpha} \|f\|_{L_x^p}. 
\end{equation*}
Since $2\alpha > 1$, we get
\begin{equation*}
\bigg\| \int_{0}^{t} e^{i(t-s)\Delta} V e^{is\Delta} f {\rm d} B_s \bigg\|_{L_{\omega}^{\rho}L_{x}^{q}} \lesssim t^{-\alpha} \|f\|_{L_x^p}. 
\end{equation*}
The third term can be controlled in the same way (but requiring $\alpha>1$). This completes the three terms in the first line of the right hand side of \eqref{eq:duhamel_2}. 

The rest of the section is devoted to the four terms in \eqref{eq:duhamel_list}. We start with the first one. 

\subsection{Term (I)}

Let
\begin{equation*}
G_{t}(r,s) = \big\| e^{i(t-s)\Delta} V^2 e^{i(s-r)\Delta} V^2 \Psi(r) \big\|_{L_{\omega}^{\rho} L_{x}^{q}}, 
\end{equation*}
so we have
\begin{equation} \label{eq:list_1}
t^{\alpha} \|\text{(I)}\|_{L_{\omega}^{\rho} L_{x}^{q}} \leq \delta^4 t^{\alpha} \int_{0}^{t} \int_{0}^{s} G_{t}(r,s) {\rm d}r {\rm d}s. 
\end{equation}
Recall that $t>2$. In the domains $r \in [0,1]$, $r \in [1,t-1]$ and $r \in [t-1,t]$, by Lemmas~\ref{le:variable_small}, ~\ref{le:op_strong_decay} and~\ref{le:exchange_trig} respectively, we can bound the integrand $G_{t}$ pointwise by
\begin{equation} \label{eq:integrand_bound}
G_{t}(r,s) \lesssim
\begin{cases}
\scal{t-s}^{-\alpha} \scal{s-r}^{-\alpha} \|f\|_{L_x^p}\;, &r\in[0,1]\\
\scal{t-s}^{-\alpha} \scal{s-r}^{-\alpha} \|\Psi(r)\|_{L_{\omega}^{\rho}L_{x}^{q}}\;, &r \in [1,t-1]\\
(t-r)^{-\alpha} \|\Psi(r)\|_{L_{\omega}^{\rho}L_{x}^{q}}\;, &r \in [t-1,t].
\end{cases}
\end{equation}
All the proportionality constants above are independent of $r$, $s$ and $t$. Note that the factor in the last bound is $(t-r)^{-\alpha}$ rather than $\scal{t-r}^{-\alpha}$, so it has a singularity at $r \approx t$. 

According to \eqref{eq:integrand_bound}, we decompose the integral on the right hand side of \eqref{eq:list_1} into three disjoint regions as
\begin{equation*}
\int_{0}^{t} \int_{0}^{s} G_t {\rm d}r {\rm d}s = \int_{0}^{t} \int_{0}^{s \wedge 1} G_t {\rm d}r {\rm d}s + \int_{1}^{t} \int_{1}^{s \wedge (t-1)} G_t {\rm d}r {\rm d}s + \int_{t-1}^{t} \int_{t-1}^{s} G_t {\rm d}r {\rm d}s. 
\end{equation*}
For the first one, using the first bound in \eqref{eq:integrand_bound} and $\alpha>1$, we have
\begin{equation} \label{eq:1_1}
\int_{0}^{t} \int_{0}^{s \wedge 1} G_t {\rm d}r {\rm d}s \lesssim \|f\|_{L_x^p} \int_{0}^{t} \int_{0}^{s \wedge 1} \scal{t-s}^{-\alpha} \scal{s-r}^{-\alpha} {\rm d}r {\rm d}s \lesssim t^{-\alpha} \|f\|_{L_x^p}, 
\end{equation}
For the second one, using the second bound in \eqref{eq:integrand_bound}, we get
\begin{equation} \label{eq:1_2}
\begin{split}
\int_{1}^{t} \int_{1}^{s \wedge (t-1)} G_t {\rm d}r {\rm d}s &\lesssim \int_{1}^{t} \int_{1}^{s \wedge (t-1)} \scal{t-s}^{-\alpha} \scal{s-r}^{-\alpha} \|\Psi(r)\|_{L_{\omega}^{\rho}L_{x}^{q}} {\rm d}r {\rm d}s\\
&\leq \sup_{r \in [1,t-1]} \Big( \scal{r}^{\alpha} \|\Psi(r)\|_{L_{\omega}^{\rho} L_{x}^{q}} \Big) \int_{1}^{t} \int_{1}^{s} \scal{t-s}^{-\alpha} \scal{s-r}^{-\alpha} \scal{r}^{-\alpha} {\rm d}r {\rm d}s\\
&\lesssim t^{-\alpha} \sup_{r \in [0,t]} \Big( r^{\alpha} \|\Psi(r)\|_{L_{\omega}^{\rho} L_{x}^{q}} \Big), 
\end{split}
\end{equation}
where in the last bound we have also used $\alpha>1$. For the third term, we have
\begin{equation} \label{eq:1_3}
\begin{split}
t^{\alpha} \int_{t-1}^{t} \int_{t-1}^{s} G_t {\rm d}r {\rm d}s &\lesssim t^{\alpha} \int_{t-1}^{t} \int_{t-1}^{s} (t-r)^{-\alpha} \|\Psi(r)\|_{L_{\omega}^{\rho}L_{x}^{q}}\\
&\lesssim \sup_{r \in [t-1,t]} \Big( r^{\alpha} \|\Psi(r)\|_{L_{\omega}^{\rho} L_{x}^{q}} \Big) \int_{t-1}^{t} \int_{t-1}^{s} (t-r)^{-\alpha} {\rm d}r {\rm d}s\\
&\lesssim \sup_{r \in [t-1,t]} \Big( r^{\alpha} \|\Psi(r)\|_{L_{\omega}^{\rho} L_{x}^{q}} \Big). 
\end{split}
\end{equation}
Here, for the second line above, we used $r \in [t-1,t]$ and $t > 2$ so that we replaced $t^{\alpha}$ with $r^{\alpha}$ with a universal proportionality constant. Also, since $\alpha<2$, the integral on the second line above is finite if and only if $\alpha<2$, which is indeed the case. Hence one obtains the last bound. 

Combining \eqref{eq:list_1}, \eqref{eq:1_1}, \eqref{eq:1_2} and \eqref{eq:1_3}, we conclude that
\begin{equation*}
t^{\alpha} \|\text{(I)}\|_{L_{\omega}^{\rho}L_{x}^{q}} \lesssim \delta^{4} \|f\|_{L_x^p} + \delta^4 \sup_{r \in [0,t]} \Big( r^{\alpha} \|\Psi(r)\|_{L_{\omega}^{\rho}L_{x}^{q}} \Big), 
\end{equation*}
which is of the form \eqref{eq:bootstrap}. This completes Term (I). 

\subsection{Term (II)}

We now turn to the second term in \eqref{eq:duhamel_list}. By Burkholder and triangle inequalities, we have the bound
\begin{equation*}
\|\text{(II)}\|_{L_{\omega}^{\rho} L_{x}^{q}} \lesssim \delta^3 \int_{0}^{t} \bigg( \int_{0}^{s} |G_t(r,s)|^{2} {\rm d}r \bigg)^{\frac{1}{2}} {\rm d}s, 
\end{equation*}
where this time $G_t$ has the expression
\begin{equation*}
G_t(r,s) = \|e^{i(t-s)\Delta} V^2 e^{i(s-r)\Delta} V \Psi(r)\|_{L_{\omega}^{\rho} L_{x}^{q}}, 
\end{equation*}
but still satisfies exactly the same bound as in \eqref{eq:integrand_bound}. Similar as before but with an inequality, we split the integral by
\begin{equation*}
\begin{split}
\int_{0}^{t} \Big( \int_{0}^{s} G_t^2 {\rm d}r \Big)^{\frac{1}{2}} {\rm d}s \leq &\int_{0}^{t} \Big( \int_{0}^{s \wedge 1} G_t^2 {\rm d}r \Big)^{\frac{1}{2}} {\rm d}s + \int_{1}^{t} \Big( \int_{1}^{s \wedge (t-1)} G_t^2 {\rm d}r \Big)^{\frac{1}{2}} {\rm d}s\\
&+ \int_{t-1}^{t} \Big( \int_{t-1}^{t} G_t^2 {\rm d}r \Big)^{\frac{1}{2}} {\rm d}s
\end{split}
\end{equation*}
The first two terms above can be controlled in exactly the same way as in Term (I) since they contain no singularity. For the third one, also similar as before, we have
\begin{equation*}
t^{\alpha} \int_{t-1}^{t} \Big( \int_{t-1}^{s} G_t^2 {\rm d}r \Big)^{\frac{1}{2}} {\rm d}s \lesssim \sup_{r \in [t-1,t]} \Big( r^{\alpha} \|\Psi(r)\|_{L_{\omega}^{\rho}L_{x}^{q}} \Big) \int_{t-1}^{t} \Big(\int_{t-1}^{s} (t-r)^{-2\alpha} {\rm d}r \Big)^{\frac{1}{2}} {\rm d}s. 
\end{equation*}
This time, the integral on the right hand side above has a worse singularity, but it is still integrable
\begin{equation*}
\frac{1}{2} (2\alpha-1) < 1 \leftrightarrow \alpha < \frac{3}{2}. 
\end{equation*}
Hence, $t^{\alpha} \|\text{(II)}\|_{L_{\omega}^{\rho}L_{x}^{q}}$ is also bounded by the right hand side of \eqref{eq:bootstrap} with $C_1$ and $C_2$ proportional to $\delta^3$. This completes the bound for Term (II). 

\begin{rmk} \label{rmk:singularity}
	As we can see, if there were no square root for the inner-integral, then bound for $\int_{t-1}^{t} \int_{t-1}^{s} G_t^2 {\rm d}r {\rm d}s$ would have a non-integrable singularity. This is precisely the case for Term (IV), for which we need to expand one more time to reduce this singularity to make it integrable. 
\end{rmk}

\subsection{Term (III)}

For Term (III), we have
\begin{equation*}
\|\text{(III)}\|_{L_{\omega}^{\rho}L_{x}^{q}} \lesssim \delta^3 \bigg[ \int_{0}^{t} \Big( \int_{0}^{s} G_t(r,s) {\rm d}r \Big)^{2} {\rm d}s \bigg]^{\frac{1}{2}}, 
\end{equation*}
where
\begin{equation*}
G_t(r,s) = \| e^{i(t-s)\Delta} V e^{i(s-r)\Delta} V^2 \Psi(r) \|_{L_{\omega}^{\rho}L_{x}^{q}}
\end{equation*}
satisfies the same bound as in \eqref{eq:integrand_bound}. Again, we split (the square of) the integral above as
\begin{equation*}
\begin{split}
\int_{0}^{t} \Big( \int_{0}^{s} G_{t} {\rm d}r \Big)^{2} {\rm d}s \lesssim \int_{0}^{t} &\Big( \int_{0}^{s \wedge 1} G_t {\rm d}r \Big)^{2} {\rm d}s + \int_{1}^{t} \Big( \int_{1}^{s \wedge (t-1)} G_t {\rm d}r \Big)^{2} {\rm d}s\\
&+ \int_{t-1}^{t} \Big( \int_{t-1}^{s} G_t {\rm d}r \Big)^{2} {\rm d}s. 
\end{split}
\end{equation*}
The first two terms on the right hand side can be controlled in the same way as in (I) and (II). The third one is also essentially the same except a slightly different but still integrable singularity. It can be controlled by
\begin{equation*}
t^{2\alpha} \int_{t-1}^{t} \Big( \int_{t-1}^{s} G_t {\rm d}r \Big)^{2} {\rm d}s \lesssim \sup_{r \in [t-1,t]} \Big( r^{2\alpha} \|\Psi(r)\|_{L_{\omega}^{\rho}L_{x}^{q}}^{2} \Big) \int_{t-1}^{t} \Big( \int_{t-1}^{s} (t-r)^{-\alpha} {\rm d}r \Big)^{2} {\rm d}s. 
\end{equation*}
The integral on the right hand side above is finite since $2(\alpha-1)<1$. Term (III) is then complete by taking square root of the above bounds. 

\subsection{Term (IV)}

As explained in Remark~\ref{rmk:singularity}, if we control Term (IV) in the same way as before, then we will end up a non-integrable singularity at $s \approx t$, so we need to expand the term one more time to make it integrable. 

Expanding $\Psi(r)$ as
\begin{equation*}
\Psi(r) = e^{ir\Delta} f - i \delta \int_{0}^{r} e^{i(r-u)\Delta} V \Psi(u) {\rm d}B_u - \frac{\delta^2}{2} \int_{0}^{r} e^{i(r-u)\Delta} V^2 \Psi(u) {\rm d}u, 
\end{equation*}
and substituting it into (IV) in \eqref{eq:duhamel_list}, we get
\begin{equation*}
\text{(IV)} = \text{(IV-i)} + \text{(IV-ii)} + \text{(IV-iii)}, 
\end{equation*}
where
\begin{equation} \label{eq:duhamel_list_2}
\begin{split}
\text{(IV-i)} &= - \delta^2 \int_{0}^{t} \int_{0}^{s} e^{i(t-s)\Delta} V e^{i(s-r)\Delta} V e^{ir\Delta} f {\rm d}B_r {\rm d}B_s,\\
\text{(IV-ii)} &= i \delta^3 \int_{0}^{t} \int_{0}^{s} \int_{0}^{r} e^{i(t-s)\Delta} V e^{i(s-r)\Delta} V e^{i(r-u)\Delta} V \Psi(u) {\rm d}B_u {\rm d}B_r {\rm d}B_s,\\
\text{(IV-iii)} &= \frac{\delta^4}{2} \int_{0}^{t} \int_{0}^{s} \int_{0}^{r} e^{i(t-s)\Delta} V e^{i(s-r)\Delta} V e^{i(r-u)\Delta} V^2 \Psi(u) {\rm d}u {\rm d}B_r {\rm d}B_s. 
\end{split}
\end{equation}
Term (IV-i) is straightforward. Indeed, by Burkholder and triangle inequalities, we have
\begin{equation*}
\|\text{(IV-i)}\|_{L_{\omega}^{\rho}L_{x}^{q}} \lesssim \delta^2 \bigg( \int_{0}^{t} \int_{0}^{s} \|e^{i(t-s)\Delta} V e^{i(s-r)\Delta} V e^{ir\Delta} f\|_{L_{\omega}^{\rho} L_{x}^{q}}^{2} {\rm d}r {\rm d}s \bigg)^{\frac{1}{2}}. 
\end{equation*}
Since $t>2$, by Lemma~\ref{le:op_strong_decay}, the integrand satisfies a pointwise bound
\begin{equation*}
\|e^{i(t-s)\Delta} V e^{i(s-r)\Delta} V e^{ir\Delta} f\|_{L_{\omega}^{\rho} L_{x}^{q}} \lesssim \scal{t-s}^{-\alpha} \scal{s-r}^{-\alpha} \scal{r}^{-\alpha} \|f\|_{L_x^p}, 
\end{equation*}
and hence the desired bound for (IV-i) follows immediately. As for (IV-ii), we have
\begin{equation*}
\|\text{(IV-ii)}\|_{L_{\omega}^{\rho}L_{x}^{q}} \lesssim \delta^3 \bigg( \int_{0}^{t} \int_{0}^{s} \int_{0}^{r} |G_t(u,r,s)|^{2} {\rm d}u {\rm d}r {\rm d}s  \bigg)^{\frac{1}{2}}, 
\end{equation*}
where again by Lemmas~\ref{le:variable_small}, ~\ref{le:op_strong_decay} and~\ref{le:exchange_trig}, the integrand
\begin{equation*}
G_t(u,r,s) = \|e^{i(t-s)\Delta} V e^{i(s-r)\Delta} V e^{i(r-u)\Delta} V \Psi(u)\|_{L_{\omega}^{\rho} L_{x}^{q}}
\end{equation*}
satisifes the pointwise bound
\begin{equation} \label{eq:integrand_bound_2}
G_t(u,r,s) \lesssim
\begin{cases}
\scal{t-s}^{-\alpha} \scal{s-r}^{-\alpha} \scal{r-u}^{-\alpha} \|f\|_{L_x^p}, &u\in[0,1]\\
\scal{t-s}^{-\alpha} \scal{s-r}^{-\alpha} \scal{r-u}^{-\alpha} \|\Psi(u)\|_{L_{\omega}^{\rho}L_{x}^{q}}, &u \in [1,t-1]\\
(t-u)^{-\alpha} \|\Psi(u)\|_{L_{\omega}^{\rho}L_{x}^{q}}\;, &u \in [t-1,t].
\end{cases}
\end{equation}
We now split the integration into three disjoint sub-domains: $\{u \leq 1\}$, $\{u \in [1,t-1]\}$ and $\{u \geq t-1\}$. The bound for the first two domains are similar as before. The only one containing a singularity is the third one, which we have the bound
\begin{equation*}
t^{2\alpha} \int_{t-1<u<r<s<t} G_t^2 {\rm d}u {\rm d}r {\rm d}s \lesssim \sup_{u \in [t-1,t]} \Big( u^{2\alpha} \|\Psi(u)\|_{L_{\omega}^{\rho}L_{x}^{q}}^{2} \Big) \int_{t-1}^{t} \int_{t-1}^{s} \int_{t-1}^{r} (t-u)^{-2\alpha} {\rm d}u {\rm d}r {\rm d}s. 
\end{equation*}
This time, the integral on the right hand side above is finite since the exponent of the singularity satisfies $2\alpha-2<1$. One can then conclude the desired bound for (IV-ii). 

Finally, the bound for Term (IV-iii) is essentially the same. This completes the case for (IV) as well as the proof of Theorem~\ref{th:main_decay} in $d=3$.

\section{Higher dimensions}
\label{sec:high_dim}

The argument for higher dimensions is essentially the same, except that when the inner-most integration variable is close to $t$, the singularity is worse since $\alpha$ becomes larger. Hence, we need to expand more times with the Duhamel formula in order to make the singularity integrable. 

In dimension $d$, we expand $\Psi(t)$ to the $d$-th order so that it can be expressed as a linear combination of terms of the form
\begin{equation} \label{eq:const_terms}
\bigg\{ \int_{0<s_1<\cdots<s_j<t} e^{i(t-s_j)\Delta} V_{j} \cdots V_2 e^{i(s_2-s_1)\Delta} V_1 e^{is_1 \Delta} f  {\rm d}X_{s_1} \cdots {\rm d}X_{s_j} \bigg\}_{j=1, \dots, d}
\end{equation}
and
\begin{equation} \label{eq:expansion}
\bigg\{\int_{0<s_1<\cdots<s_d<t} e^{i(t-s_d)\Delta} V_{d} \cdots V_{2} e^{i(s_2-s_1)\Delta} V_1 \Psi(s_1) {\rm d}X_{s_1} \cdots {\rm d}X_{s_d} \bigg\}, 
\end{equation}
where each $X_{s_j}$ is either $s_j$ or $B_{s_j}$, and each $V_j$ is either $\delta V$ or $\delta^2 V^2$. To control these terms, we first note that if $q$ is sufficiently close to $1$, then
\begin{equation*}
\alpha = d\Big( \frac{1}{2} - \frac{1}{q} \Big) = \frac{d}{2} - \kappa
\end{equation*}
for some sufficiently small $\kappa>0$. 

The bound for the terms in \eqref{eq:const_terms} are straightforward. As for \eqref{eq:expansion}, when $s_1<t-1$, the proof are exactly the same as before since the pointwise decay estimates are even better. Singularity occurs when $s_1 > t-1$, and the worst case is that all $X_{s_j}$ in \eqref{eq:expansion} are $B_{s_j}$, and we need to control
\begin{equation} \label{eq:high_dim_worst}
\delta^{d} t^{\alpha} \bigg( \int_{t-1<s_1<\cdots<s_d<t} \|e^{i(t-s_d)\Delta} V \cdots V e^{i(s_2-s_1)\Delta} V \Psi(s_1)\|_{L_{\omega}^{\rho}L_{x}^{q}}^{2} {\rm d}s_1 \cdots {\rm d}s_d \bigg)^{\frac{1}{2}}. 
\end{equation}
By Lemma~\ref{le:exchange_trig}, we have the pointwise bound
\begin{equation*}
\|e^{i(t-s_d)\Delta} V \cdots V e^{i(s_2-s_1)\Delta} V_1 \Psi(s_1)\|_{L_{\omega}^{\rho}L_{x}^{q}}^{2} \lesssim (t-s_1)^{-2\alpha} \|\Psi(s_1)\|_{L_{\omega}^{\rho}L_{x}^{q}}^{2}
\end{equation*}
for the integrand, so that \eqref{eq:high_dim_worst} is controlled by
\begin{equation*}
\delta^d \sup_{r \in [t-1,t]} \Big( r^{\alpha} \|\Psi(r)\|_{L_{\omega}^{\rho}L_{x}^{q}}\Big) \cdot \bigg( \int_{t-1<s_1<\cdots<s_d<t} (t-s_1)^{-2\alpha} {\rm d}s_1 \cdots {\rm d}s_k \bigg)^{\frac{1}{2}}. 
\end{equation*}
The integral on the right hand side above is finite since $-2\alpha+(d-1) > -1$. This bound is of the form \eqref{eq:bootstrap} with both $C_2$ proportional to $\delta^d$. The proof is then complete.

\bibliographystyle{Martin}
\bibliography{Refs}

\begin{thebibliography}{HRZ18}
\expandafter\ifx\csname url\endcsname\relax
  \def\url#1{\texttt{#1}}\fi
\expandafter\ifx\csname urlprefix\endcsname\relax\def\urlprefix{URL }\fi
\expandafter\ifx\csname href\endcsname\relax
  \def\href#1#2{#2}\fi
\expandafter\ifx\csname burlalt\endcsname\relax
  \def\burlalt#1#2{\href{#2}{\texttt{#1}}}\fi

\bibitem[BL12]{Bergh_interpolation}
\textsc{J.~Bergh} and \textsc{J.~L{\"o}fstr{\"o}m}.
\newblock \emph{Interpolation {S}paces: {A}n {I}ntroduction}.
\newblock Springer Science \& Business Media, 2012.

\bibitem[BP99]{BP}
\textsc{Z.~Brze\'{z}niak} and \textsc{S.~Peszat}.
\newblock Space-time continuous solutions to {SPDE}'s driven a homogeneous
  {W}iener process.
\newblock \emph{Stud. Math.} \textbf{137}, no.~3, (1999), 261--290.

\bibitem[Brz97]{Brzezniak}
\textsc{Z.~Brze\'{z}niak}.
\newblock On stochastic convolution in {B}anach spaces and applications.
\newblock \emph{Stochast. Stochast. Rep.} \textbf{61}, (1997), 245--295.

\bibitem[BRZ14]{Rockner1}
\textsc{V.~Barbu}, \textsc{M.~R\"{o}ckner}, and \textsc{D.~Zhang}.
\newblock Stochastic nonlinear {S}chr\"{o}dinger equations with linear
  multiplicative noise: rescaling approach.
\newblock \emph{J. Nonlinear Sci.} \textbf{24}, no.~3, (2014), 383--409.

\bibitem[BRZ16]{Rockner2}
\textsc{V.~Barbu}, \textsc{M.~R\"{o}ckner}, and \textsc{D.~Zhang}.
\newblock Stochastic nonlinear {S}chr\"{o}dinger equations.
\newblock \emph{Nonlinear Anal. Theory Methods Appl.} \textbf{136}, (2016),
  168--194.

\bibitem[Bur73]{Burkholder}
\textsc{D.~Burkholder}.
\newblock Distribution function inequalities for martingales.
\newblock \emph{Ann. Prob.} \textbf{1}, no.~1, (1973), 19--42.

\bibitem[Cal63]{calderon1963intermediate}
\textsc{A.-P. Calder{\'o}n}.
\newblock Intermediate spaces and interpolation.
\newblock \emph{Stud. Math.} \textbf{1}, (1963), 31--34.

\bibitem[Cal64]{calderon1964intermediate}
\textsc{A.-P. Calder{\'o}n}.
\newblock Intermediate spaces and interpolation, the complex method.
\newblock \emph{Stud. Math.} \textbf{24}, (1964), 113--190.

\bibitem[Cal66]{calderon1966spaces}
\textsc{A.-P. Calder{\'o}n}.
\newblock Spaces between $l^1$ and $l^\infty$ and the theorem of marcinkiewicz.
\newblock \emph{Stud. Math.} \textbf{26}, (1966), 273--299.

\bibitem[Caz03]{cazenave2003semilinear}
\textsc{T.~Cazenave}.
\newblock \emph{Semilinear {S}chr{\"o}dinger equations}.
\newblock Courant Lecture Notes in Mathematics. American Mathematical Society,
  2003.

\bibitem[dBD99]{dBD}
\textsc{A.~de~Bouard} and \textsc{A.~Debussche}.
\newblock A stochastic nonlinear {S}chr\"{o}dinger equation with multiplicative
  noise.
\newblock \emph{Comm. Math. Phys.} \textbf{205}, no.~1, (1999), 161--181.

\bibitem[dBD03]{Debussche_H1}
\textsc{A.~de~Bouard} and \textsc{A.~Debussche}.
\newblock The stochastic nonlinear {S}chr\"odinger equation in {$H^{1}$}.
\newblock \emph{Stoch. Anal. Appl.} \textbf{21}, no.~1, (2003), 97--126.

\bibitem[FX18a]{snls_mass_critical}
\textsc{C.~Fan} and \textsc{W.~Xu}.
\newblock Global well-posedness for the defocusing mass-critical stochastic
  nonlinear {S}chr\"{o}dinger equation on $\mathbb{R}$ at {$L^2$} regularity.
\newblock \emph{ArXiv e-prints} (2018).
\newblock \burlalt{arXiv:1810.07925}{http://arxiv.org/abs/1810.07925}.

\bibitem[FX18b]{snls_subcritical_approx}
\textsc{C.~Fan} and \textsc{W.~Xu}.
\newblock Subcritical approximations to stochastic defocusing mass-critical
  nonlinear {S}chr\"odinger equation on $\mathbb{R}$.
\newblock \emph{ArXiv e-prints} (2018).
\newblock \burlalt{arXiv:1810.09407}{http://arxiv.org/abs/1810.09407}.

\bibitem[Hor18]{hornung2016nonlinear}
\textsc{F.~Hornung}.
\newblock The nonlinear stochastic {S}chr{\"o}dinger equation via stochastic
  {S}trichartz estimates.
\newblock \emph{J. Evol. Equ.} \textbf{18}, no.~3, (2018), 1085--1113.

\bibitem[HRZ18]{Rockner_finite_scatter}
\textsc{S.~Herr}, \textsc{M.~R\"ockner}, and \textsc{D.~Zhang}.
\newblock Scattering for stochastic nonlinear {S}chr\"odinger equations.
\newblock \emph{ArXiv e-prints} (2018).
\newblock \burlalt{arXiv:1804.10429}{http://arxiv.org/abs/1804.10429}.

\bibitem[JSS91]{decay_deterministic}
\textsc{J.-L. Journ\'e}, \textsc{A.~Soffer}, and \textsc{C.~Sogge}.
\newblock Decay estimates for {S}chro\"odinger operators.
\newblock \emph{Comm. Pure Appl. Math.} \textbf{44}, no.~5, (1991), 573--604.

\bibitem[Sch07]{dispersive_survey_Schlag}
\textsc{W.~Schlag}.
\newblock Dispersive estimates for schr\"odinger operators: a survey.
\newblock \emph{Ann. Math. Stud.} \textbf{163}, (2007), 255--285.

\bibitem[Tao06]{tao2006nonlinear}
\textsc{T.~Tao}.
\newblock \emph{Nonlinear dispersive equations: local and global analysis}.
\newblock CBMS Regional Conference Series in Mathematics. American Mathematical
  Society, 2006.

\bibitem[Zha18]{snls_critical_Zhang}
\textsc{D.~Zhang}.
\newblock Stochastic nonlinear {S}chr\"odinger equations in the defocusing mass
  and energy critical cases.
\newblock \emph{ArXiv e-prints} (2018).
\newblock \burlalt{arXiv:1811.00167}{http://arxiv.org/abs/1811.00167}.

\end{thebibliography}

\end{document}